\documentclass[11pt, reqno]{amsart}
\usepackage{multicol,a4wide}
\usepackage{amsmath}
\usepackage{amsaddr}
\usepackage{color}
\usepackage{cmap,mathtools}
\usepackage{cite}
\usepackage{xcolor}
\usepackage{enumerate}
\usepackage{epsf}
\usepackage{amsbsy,amsmath}
\usepackage{mathtools}
\usepackage{mathrsfs}
\usepackage{amsfonts}
\usepackage{amssymb}
\usepackage{eucal}
\usepackage{graphics,mathrsfs}

\usepackage{amsthm}
\usepackage{xcolor}
\usepackage{hyperref}
\hypersetup{
    colorlinks=true,
    linkcolor=red,
    filecolor=magenta,      
    urlcolor=green,
    citecolor=forestgreen,
}
\definecolor{forestgreen}{rgb}{0.13, 0.85, 0.15}
\usepackage{orcidlink}

\makeatletter
\newcommand{\leqnomode}{\tagsleft@true\let\veqno\@@leqno}
\usepackage{yhmath,amsmath,amsfonts,amssymb,enumerate,amsthm,appendix,caption,lipsum,}
\usepackage[T1]{fontenc}
\usepackage{enumerate}
\usepackage{mathtools,bm}
\usepackage{cmap}
\usepackage{bigints}
\usepackage{esint}
\usepackage[margin=0.90in]{geometry}

\newtheorem{definition}{Definition}[section]
\newtheorem{theorem}[definition]{Theorem}

\newtheorem{remark}[definition]{Remark}

\newtheorem{lemma}[definition]{Lemma}

\numberwithin{equation}{section}

\DeclarePairedDelimiter\norm{\lVert}{\rVert}
\makeatletter
\let\oldnorm\norm
\def\norm{\@ifstar{\oldnorm}{\oldnorm*}}

\makeatletter

\DeclareMathAlphabet{\mathpzc}{T1}{pzc}{m}{it}

\title[On weak and viscosity solutions to nonhomogeneous equation]{On weak and viscosity solutions to a nonhomogeneous mixed local-nonlocal equation}
\author[R. Lakshmi and S. Ghosh]{R. Lakshmi$^1$\orcidlink{0009-0009-5796-1009} and Sekhar Ghosh$^{1,*}$\orcidlink{0000-0002-5082-2374}} 
\thanks{$^*$Corresponding author}

\begin{document}
\maketitle
\centerline{$^1$Department of Mathematics, National Institute of Technology Calicut,}
\centerline{Kozhikode - 673601, Kerala, India.}
\centerline{Email: lakshmir1248@gmail.com, sekharghosh1234@gmail.com/sekharghosh@nitc.ac.in}

\begin{abstract} 
This paper explores the relationship between weak and viscosity solutions to a nonhomogeneous mixed local and non-local $p$-Laplace equation in a bounded Lipschitz domain in $\mathbb{R}^N$. Under certain conditions, we derive the comparison principle for weak subsolutions and weak supersolutions to the problem. For $1<p<\infty$, we establish that continuous weak supersolutions to the problem are viscosity supersolutions, using the comparison principle. Furthermore, we show that bounded viscosity supersolutions are weak supersolutions for $p \geq 2$.

\end{abstract}
\maketitle

	\section{Introduction and Main Results}
Equivalence between different notions of solutions to partial differential equations (PDEs) has been a prominent area of study due to its applications in various problems such as free boundary problems and approximation problems of PDEs (see \cite{JLP10, MO23, BLO20, IN10}).  The relation between different notions of solutions to problems also help to study the existence and associated properties of solutions. The investigation of equivalence between solutions to PDEs was initiated in the linear case with the Laplacian by Lions \cite{L83}, Ishii \cite{I95} (see also \cite{Lind16, CIL92}). Consider the problem involving the $p$-Laplacian
\begin{equation}\label{p-lap}
    -\Delta_p u:= \operatorname{div}(|\nabla u|^{p-2}\nabla u)=f(x,u,\nabla u) \text{ in } \Omega
\end{equation}
where $\Omega \subset \mathbb{R}^N$ is a bounded open set, $1<p<\infty$ and $\Delta_p$ is the $p$-Laplacian. The equivalence of weak and viscosity solutions to \eqref{p-lap} was established by Juutinen \textit{et al.} \cite{JLM01} under the homogeneous data $f\equiv 0$. The authors in \cite{JLM01} also proved that different versions of solutions coincide for the homogeneous parabolic $p$-Laplace equation. The authors in \cite{JLM01} used the concept of $p$-subharmonic and superharmonic functions to demonstrate the equivalence of solutions. Subsequently, a different proof is obtained for the equivalence of weak and viscosity solutions to \eqref{p-lap} by Julin and Juutinen \cite{JJ12} for $f \equiv 0$. In the same paper \cite{JJ12}, the equivalence of weak and viscosity solutions is also discussed for the nonhomogeneous case with $f:=f(x)$. Medina and Ochoa \cite{MO19} extended the result further, proving the equivalence between weak and viscosity solutions to \eqref{p-lap} for $f:=f(x,u,\nabla u(x))$ with certain growth restrictions. Moreover, the analysis of the relationship between different types of solutions has been extended to problems similar to \eqref{p-lap} involving $p(x)$-Laplacian (see Siltakoski \cite{S18} and Medina and Ochoa \cite{MO23}). It is noteworthy to mention that Fang \textit{et al.} \cite{FRZ24} explored the equivalence and regularity of solutions to a non-homogeneous double phase problem 
\begin{equation}\label{local DP}
    \operatorname{div}(|\nabla u|^{p-2}\nabla u+a(x)|\nabla u|^{q-2}\nabla u)=f(x,u,\nabla u) \text{ in }\Omega,
\end{equation}
with $a(x)\geq 0$ and $1 <p \leq q<\infty$ in a bounded domain $\Omega$. For further developments on the theory of viscosity solutions, we refer to \cite{Lind19, CIL92, K04, LM07} and the references therein.
\par Next, we move our attention to discuss the developments in studying the relation between various notions of solutions to the fractional problem
\begin{equation}\label{frac p-lap}
    (-\Delta_p)^s u=f(x,u,D_s^p u) \text{ in } \Omega,
\end{equation}
where $\Omega$ is a bounded open subset of $\mathbb{R}^N$, $0<s<1<p<\infty$, $(-\Delta_p)^s$ is the fractional $p$-Laplacian defined by
\begin{equation*}
    (-\Delta_p)^s u(x):=2\operatorname{P.V.}\int_{\mathbb{R}^N}\frac{|u(x)-u(y)|^{p-2}(u(x)-u(y))}{|x-y|^{N+sp}}dy.
\end{equation*}
$D_s^p u$ is the fractional gradient given by
\begin{equation*}
    D_s^p u(x):=\int_{\mathbb{R}^N}\frac{|u(x)-u(y)|^p}{|x-y|^{N+sp}}dy.
\end{equation*}
For the homogeneous case $f \equiv 0$, of \eqref{frac p-lap}, Korvenpää \textit{et al.} \cite{KKL19} guaranteed that weak, viscosity and $p$-harmonic solutions coincide. Later, Barrios and Medina \cite{BM21} derived that the weak and viscosity solutions are the same for \eqref{frac p-lap} with $f:=f(x,u, D_s^p u(x))$ by imposing some conditions on $f$. The analysis of the relation between solutions to a nonlocal analogue of the double phase problem \eqref{local DP} is investigated for the homogeneous case by Fang and Zhang \cite{FZ23} and for the nonhomogeneous case by Ghosh et al. \cite{GLZ26}.

\par Recently, there have also been studies on the equivalence of solutions to problems involving the mixed local and nonlocal operator of the following type: 
\begin{equation}\label{mixed p}
    \mathfrak{L}_{s,p}:=-\Delta_p+(-\Delta_p)^s.
\end{equation} 
The equivalence between weak and viscosity solutions involving $\mathfrak{L}_{s,p}$ has been derived in the homogeneous case by Lakshmi and Ghosh \cite{LG2026}. Shang and Zhang \cite{SZ23} have also explored a relation between solutions to a particular case of the nonhomogeneous equation involving $\mathfrak{L}_{s,p}$.

\par In this work, our objective is to study the relation between weak and viscosity solutions to the following mixed local and nonlocal problem in a bounded domain $\Omega \subset \mathbb{R}^N$ having Lipschitz boundary $\partial \Omega$, given by
\begin{align}\label{G}
   \mathfrak{L}_{s,p}u &=f(x,u, \nabla u, D_s^p u) \text{ in } \Omega,
\end{align}
where $0<s<1<p<\infty$. The operator $\mathfrak{L}_{s,p}$ is given by \eqref{mixed p} where we take $(-\Delta_p)^s$ to be a generalized version of the fractional $p$-Laplacian given by
\begin{equation}\label{gen frac p}
    (-\Delta_p)^s u(x):=2\operatorname{P.V.}\int_{\mathbb{R}^N}{|u(x)-u(y)|^{p-2}(u(x)-u(y))}{K_{s,p}(x,y)}dy.
\end{equation}
The kernel $K_{s,p}:\mathbb{R}^N \times \mathbb{R}^N \rightarrow \mathbb{R}$ is a measurable function satisfying the following conditions:
\begin{enumerate}[(i)]
    \item For all $x,y \in \mathbb{R}^N$, we have $K_{s,p}(x,y)=K_{s,p}(y,x).$ 
    \item $K_{s,p}(x+z,,y+z)=K_{s,p}(x,y)$ for all $x,y,z \in \mathbb{R}^N$.
    \item There exists $\Lambda >0$ such that for all $x\neq y \in \mathbb{R}^N$, we have
    \begin{equation}\label{G-kernel}
        \frac{1}{\Lambda} \frac{1}{| x-y |^{N+ps}} \leq K_{s,p}(x,y) \leq \Lambda \frac{1}{| x-y |^{N+ps}}.
    \end{equation} 
    \item The map $x \mapsto K_{s,p}(x,y)$ is continuous in $\mathbb{R}^N \setminus \{y\}$ for every $y \in \mathbb{R}^N $.
\end{enumerate}
\par An important tool required to establish the relation of weak and viscosity solutions to PDEs is the comparison principle. One may refer to \cite{MO19, BM21, DFR19, LG2026} etc. for some proofs of comparison principles obtained in local, nonlocal and mixed local-nonlocal problems. Before introducing the main theorems, we first give the following definition of comparison principle.
\begin{definition}\label{G-CP DFN}
    Let $\Omega' \subset \Omega$ be an open set. A weak supersolution $u$ of \eqref{G} in $\Omega'$ satisfies the comparison principle in $\Omega'$ if for every weak subsolution $w$ of \eqref{G} in $\Omega'$ satisfying $u \geq w$ in $\mathbb{R}^N \setminus \Omega'$, then we have $u \geq w$ in $\mathbb{R}^N$.
\end{definition}
We proved the comparison principle for weak supersolutions to \eqref{G}, on imposing certain restrictions on $f$ in Section \ref{G-s3}. Next, we provide the statement of our first main result.
\begin{theorem}\label{G-T1}
    Let $p\geq 2$, $f:=f(x,t,\eta,\zeta)$ be continuous with respect to $x,t,\eta$ and $\zeta$ and Lipschitz continuous with respect to $\eta$ and $\zeta$. Also, assume that the comparison principle (Definition \ref{G-CP DFN}) holds true in any open set $\Omega' \subset \Omega$. Then, every continuous weak supersolution to \eqref{G} is a viscosity supersolution to \eqref{G}. 
\end{theorem}
Due to the lack of continuity of $\mathfrak{L}_{s,p} u$ in the range $1<p<2$ even for smooth functions $u$ at points of vanishing gradients, we could not extend Theorem \ref{G-T1} to the case $1<p<2$. However, we have attained that under a different set of assumptions, the weak solutions to \eqref{G} are viscosity solutions to the same, which is stated as below.
\begin{theorem}\label{G-T3}
    Let $1<p<2$, and let $f:=f(x,t,\eta,\zeta)$ be continuous with respect to $x,t,\eta$, uniformly continuous with respect to $\zeta$ and non-decreasing in $t$. Also, assume that the comparison principle (Definition \ref{G-CP DFN}) holds in any open set $\Omega' \subset \Omega$. Then, every continuous weak supersolution to \eqref{G} is a viscosity supersolution to \eqref{G}. 
\end{theorem}
Finally, we also prove the following result on bounded viscosity solutions and weak solutions to \eqref{G}.
\begin{theorem}\label{G-T2}
      Let $2\leq p<\infty$, $f:=f(x,t,\eta,\zeta)$ be a uniformly continuous function which is non-increasing in t, Lipschitz continuous in $\eta$ and $\zeta$. Also, assume that there exist continuous functions $g_1,g_2:\mathbb{R}\rightarrow \mathbb{R}$, and a function $g_3\in L^\infty_{\text{loc}}(\Omega)$ such that
\begin{equation}\label{G-f condn}
   |f(x,t,\eta,\zeta)|\leq g_1(|t|)|\eta|^{p-1}+g_2(|t|)|\zeta|^{\frac{p-1}{p}}+g_3(x), (x,t,\eta,\zeta)\in \Omega \times \mathbb{R}\times \mathbb{R}^N \times \mathbb{R}.
\end{equation}
Let $u$ be a bounded viscosity supersolution to \eqref{G}. Then, $u$ is a weak supersolution to \eqref{G}.
\end{theorem}
We have not been able to generalize Theorem \ref{G-T2} to the case $1<p<2$ due to the varying behaviour and singulairites of the $p$-Laplacian part and fractional $p$-Laplacian part of $\mathfrak{L}_{s,p}$ in this range (see \cite[Remark 1.2]{LG2026} for more details.).
 \par We conclude this section by providing an overview of the structure of the paper. In Section \ref{G-s2}, we discuss the basic definitions, notations and results required to carry out our study. In Section \ref{G-s3}, we prove the comparison principle for a subcase of \eqref{G} and prove Theorem \ref{G-T1} and Theorem \ref{G-T3}. Lastly, in Section \ref{G-T2}, we establish Theorem \ref{G-T2}.

\section{Preliminaries tools and function space setup}\label{G-s2}
\noindent We begin this section by recalling certain fundamental definitions, which are prerequisites to our study. 
Recall the definitions of the Sobolev space and fractional Sobolev spaces (see \cite{DD12, DPV12, L23}). For $1\leq p<\infty$, the Sobolev space $W^{1,p}(\Omega)$ is given by
\begin{equation*}
    W^{1,p}(\Omega)=\{u \in L^p(\Omega): \frac{\partial u}{\partial x_i} \in L^p(\Omega) \text{ for } i\in[1, N] \},  \quad p\in [1,\infty),
\end{equation*}
where $\Omega \subset \mathbb{R}^N, \ N\geq2$, is an open set. The space $W^{1,p}(\Omega)$ is endowed with the norm $\|u\|_{W^{1,p}(\Omega)}=\Big( \|u\|_{L^p(\Omega)}^p +\|\nabla u\|_{L^p(\Omega)}^p \Big)^{\frac{1}{p}}$.
For $0<s<1\leq p<\infty$, the fractional Sobolev space $W^{s,p}(\Omega)$ is defined by
\begin{equation*}
    W^{s,p}(\Omega)=\Big\{u \in L^p(\Omega): \int_\Omega \int_\Omega \frac{| u(x)-u(y)|^p}{| x-y |^{N+ps}}dxdy <\infty\Big\},
\end{equation*}
endowed with the norm $\|u\|_{W^{s,p}(\Omega)}=\left( \|u\|_{L^p(\Omega)}^p +[u]_{W^{s,p}(\Omega)}^p \right)^{\frac{1}{p}}$, where $[\cdot]_{W^{s,p}(\Omega)}$ is the Gagliardo semi-norm defined by
\begin{equation*}
    [u]_{W^{s,p}(\Omega)}^p=\int_\Omega \int_\Omega \frac{| u(x)-u(y)|^p}{| x-y |^{N+ps}}dxdy.
\end{equation*}
Before defining viscosity solutions, we first present the notion of tail spaces (see \cite{DKP16}). For $0<s<1<p<\infty$, the \textit{Tail space} $L_{s,p}(\mathbb{R}^N)$ is defined by
\begin{equation*}
    L_{s,p}^{p-1}(\mathbb{R}^N)=\left\{h\in L^{p-1}_{\text{loc}}(\mathbb{R}^N) : \int_{\mathbb{R}^N}\frac{|h(x)|^{p-1}}{(1+|x|)^{N+sp}}dx<\infty \right\},
\end{equation*}
and the Tail of a function $v$ concerning the ball $B_r(x_0)$ is given by
$$\operatorname{Tail}(h;z,r)=\left(r^{sp}\int_{\mathbb{R}^N \setminus B_r(z)}\frac{|h(x)|^{p-1}}{|x-z|^{N+sp}}dx\right)^{\frac{1}{p-1}}.$$
Clearly, $\operatorname{Tail}(h;x_0,r)$ is well-defined and finite for a function $h\in  L_{s,p}^{p-1}(\mathbb{R}^N)$ at all points $z\in \mathbb{R}^N$ and $r>0$.

\par To formulate the notion of viscosity solutions for \eqref{G}, an appropriate class of test functions is required. In particular, when $1<p\leq \frac{2}{2-s}$, the nonlocal operator $\mathfrak{L}_{s,p}\psi$ may fail to be well defined for arbitrary $C^2$ functions. To overcome this difficulty, one works with the class of $C_\beta^2$-functions, which ensures the well-definedness of the nonlocal term; see \cite[Section 2]{KKL19} for further details. We now recall the definition of the class $C_\beta^2$.
\begin{equation}\label{C_beta^2}
    C_\beta^2(A)=\left\{ h \in C^2(A): \sup\limits_{x\in A} \left( \frac{\min\{d_h(x),1\}^{\beta-1}}{|\nabla h(x)|} +\frac{|D^2h(x)}{d_h(x)^{\beta-2}} \right)<\infty\right\},
\end{equation}
where $N_h=\{x\in\Omega:\nabla h(x)=0\}$ denotes the critical set of $h$, and $d_h(x):=\operatorname{dist}(x,N_h)$ is the distance from $x$ to $N_h$. Motivated by the notions of viscosity solutions developed in \cite{KKL19,MO19,BM21,SZ23,LG2026}, we now introduce the concept of viscosity solutions for problem \eqref{G}.
\begin{definition}\label{G-VS}
    A function $u:\mathbb{R}^N \rightarrow [-\infty, \infty]$, which is lower semi-continuous (upper semi-continuous), is called a viscosity supersolution (subsolution) to the problem \eqref{G} if $u$ satisfies the conditions listed below.
    \begin{enumerate}[(a)]
        \item $u_- (u_+) \in L_{s,p}^{p-1}(\mathbb{R}^N)$, where $u_-=\max\{0,-u\}$.
        \item $u<\infty$ ($u>-\infty$) a.e. in $\mathbb{R}^N$ and $u>-\infty$ ($u<\infty$) in $\Omega$.
        \item If $B_r(x_0) \subset \Omega$ and $\psi \in C^2(B_r(x_0)) \cap L_{s,p}^{p-1}(\mathbb{R}^N)$ with {$\psi \leq u$ ($\psi \geq u$) in $\mathbb{R}^N$,} $\psi(x_0)=u(x_0)$ and if one of the following conditions hold for $p\in(1,2)$:
        \begin{enumerate}[(i)]
            \item if $p\in (1,2)$, then $|\nabla \psi(x)|\neq 0$ whenever $x\in B_r(x_0)\setminus \{x_0\}$,
            \item if $p\in (1, \frac{2}{2-s}]$ and $\nabla \psi(x_0)|=0$, then there exists $\beta>\frac{sp}{p-1}$ such that $\psi\in C_\beta^2(B_r(x_0))$,
        \end{enumerate}
        then $$\lim\limits_{r\rightarrow 0}\sup\limits_{x\in B_r(x_0)\setminus \{x_0\}}-\Delta_p \psi(x)+(-\Delta_p)^s \psi(x_0) \geq (\leq ) f(x_0,\psi(x_0),\nabla \psi(x_0), D_s^p\psi(x_0)).$$
    \end{enumerate}
    A function $u$ is a viscosity solution to the problem \eqref{G} if it is both a viscosity subsolution and supersolution to \eqref{G}.
 \end{definition}
 \begin{remark}\label{G-VS-R}
     The condition $\psi \leq u$ in Definition \ref{G-VS} can be replaced with the combination of the two conditions $\psi \leq u$ and $\psi<u$ in a neighbourhood of $x_0$.
 \end{remark}
To study weak solutions of equations involving the operator $\mathfrak{L}_{s,p}$, which combines both local and nonlocal components, it is necessary to introduce an appropriate space of test functions. Throughout the remainder of this paper, we assume that $\Omega\subset\mathbb{R}^N$ is a bounded open domain and that $0<s<1<p<\infty$. The space $\mathbb{X}_0^{s,p}(\Omega)$ (see \cite{BMV24, LGG24, LG2026}) is given by
\begin{equation*}
     \mathbb{X}_0^{s,p}(\Omega) = \{ w \in W^{1,p}(\mathbb{R}^N): \ w =0 \text{ in } \mathbb{R}^N \setminus \Omega \},
\end{equation*}
endowed with the norm
\begin{equation}\label{nm}
    \|w\|_{\mathbb{X}_0^{s,p}(\Omega)}
    =\left(\int_{\Omega}| \nabla w |^p dx+\int_{\mathbb{R}^N}\int_{\mathbb{R}^N}\frac{| w(x)-w(y)|^p}{| x-y |^{N+ps}} dxdy\right)^\frac{1}{p}.
\end{equation}
The Sobolev space $\mathbb{X}_0^{s,p}(\Omega)$ is separable for all $1\leq p< \infty$ and reflexive for $1< p<\infty$. Observe that using \eqref{G-kernel}, the norm $\|.\|_{\mathbb{X}_0^{s,p}(\Omega)}$ in $\mathbb{X}_0^{s,p}(\Omega)$ is equivalent to the norm $\|.\|'_{\mathbb{X}_0^{s,p}(\Omega)}$ given by
    $$\|w\|'_{\mathbb{X}_0^{s,p}(\Omega)}
    =\left(\int_{\Omega}| \nabla w |^p dx+\int_{\mathbb{R}^N}\int_{\mathbb{R}^N}{| w(x)-w(y)|^p}{K_{s,p}(x,y)} dxdy\right)^\frac{1}{p}.$$
 Also, from \cite[Lemma 2.3]{GarainKinnunen}, the norm in \eqref{nm} is equivalent to the norm given by $\|\nabla u\|_{L^p(\Omega}$. Therefore, $\mathbb{X}_0^{s,p}(\Omega)$ coincides with the closure of $C_c^{\infty}(\Omega)$ under the norm \eqref{nm}. Now, we proceed to the definition of weak solutions to \eqref{G}.
\begin{definition}\label{G-WS}
    Let $\Omega \subset \mathbb{R}^N$ be a bounded domain with Lipschitz boundary $\partial \Omega$ and $0<s<1<p<\infty$. A {function {$u \in W^{1,p}_{\text{loc}}(\Omega)\cap L_{s,p}^{p-1}(\mathbb{R}^N)$}} is a weak supersolution (subsolution) to \eqref{G} in $\Omega$ if for every $v \in \mathbb{X}_0^{s,p}(\Omega)$ with $v \geq 0$, we have
    \begin{align}\label{G-WSE}
        \int_{\Omega}|\nabla u|^{p-2}\nabla u \cdot \nabla v dx+ \int_{\mathbb{R}^N}\int_{\mathbb{R}^N} h(u(x)-u(y))(v&(x)-v(y))K_{s,p}(x,y) dydx \nonumber\\
        &\geq (\leq) \int_{\Omega}f(x,u,\nabla u,D_s^p u)v dx.
    \end{align}
    A function $u$ is called a weak solution if it is both a weak subsolution and supersolution to \eqref{G}. Equivalently, $u$ is a weak solution to \eqref{G} if equality holds in \eqref{G-WSE} for every $v \in \mathbb{X}_0^{s,p}(\Omega)$.
\end{definition} 
Next, we discuss some notations that will be used throughout the rest of the work.
\begin{remark}
 For $U \subset\mathbb{R}^N$, we denote
    \begin{itemize}
    \item For $\epsilon>0$,
    $U_\epsilon=\{x\in U: \operatorname{dist}(x,\partial U)>\epsilon\}.$
        \item $A(U)=\left( \mathbb{R}^N\right)^2 \setminus \left( \mathbb{R}^N \setminus U\right)^2= \left( U \times \mathbb{R}^N \right) \cup \left( (\mathbb{R}^N \setminus U) \times U\right).$
        \item $h(t)=|t|^{p-2}t, \ t\in \mathbb{R}$. 
    \end{itemize}
    For {$v\in W^{1,p}_{\text{loc}}(\mathbb{R}^N)$ and $w \in \mathbb{X}_0^{s,p}(\Omega)$}, the following notation will be used in the rest of the paper. 
    \begin{align*}
        H_{s,p,\Omega}(v,w)&=\int_{\Omega}|\nabla v|^{p-2}\nabla v \cdot \nabla w dx+ \int_{\mathbb{R}^N}\int_{\mathbb{R}^N} h(v(x)-v(y))(w(x)-w(y))K_{s,p}(x,y) dydx.
    \end{align*}
\end{remark}
Next, we present the notion of infimal convolution, which plays a crucial role in establishing that bounded viscosity supersolutions of \eqref{G} are also weak supersolutions.
\begin{definition}\label{G-def1}
    Let $q=2$ when $p>\frac{2}{2-s}$ and $q>\max\{2,\frac{sp}{p-1}\}$ if $1<p \leq \frac{2}{2-s}$. For a function $w:\mathbb{R}^N \rightarrow \mathbb{R}$ and $\epsilon>0$, we define the infimal convolution of $w$, denoted by $w_{\epsilon}:\mathbb{R}^N \rightarrow \mathbb{R}$ as
    \begin{equation*}
        w_{\epsilon}(x)=\inf\limits_{y \in \mathbb{R}^N} \left( w(y)+\frac{|x-y|^q}{q \epsilon^{q-1}}\right).
    \end{equation*}
\end{definition}
We conclude this section by stating the following characterization of infimal convolutions from \cite[Lemma 3.1(i)]{BM21}, which will be frequently used in Section \ref{G-s4}.
\begin{lemma}\label{G-def1-R}
    For a bounded, lower semi-continuous function $w:\mathbb{R}^N \rightarrow \mathbb{R}$ and $\epsilon>0$, there exists $r(\epsilon)>0$ such that $r(\epsilon) \rightarrow 0$ as $\epsilon \rightarrow 0$ and
    \begin{equation*}
        w_{\epsilon}(x)=\inf\limits_{y \in B_{r(\epsilon)}(x)} \left( w(y)+\frac{|x-y|^q}{q \epsilon^{q-1}}\right), \ x\in \mathbb{R}^N. 
    \end{equation*}
\end{lemma}

\section{Weak solutions are Viscosity solutions}\label{G-s3}
\noindent This section is devoted to the proof of Theorem \ref{G-T1} and Theorem \ref{G-T3}. To proceed to prove both the theorems, one common hypothesis necessary is the comparison principle. Note that for the homogeneous case, this has been obtained by the authors in \cite[Lemma 4.2]{LG2026}. We first discuss a subcase of \eqref{G} with some conditions on $f$, for which the Definition \ref{G-CP DFN} holds.
\begin{lemma}\label{G-L1}
    Let $\Omega\subset\mathbb{R}^N$ and $1<p<\infty$. Also, let $f:=f(x,t)$ be non-increasing in $t$.  Assume that $u$ and $v$ are weak subsolution and supersolution respectively, to \eqref{G} such that $u\leq v$ a.e. in $\mathbb{R}^N\setminus \Omega$. Then, $u \leq v$ a.e. in $\Omega$.
\end{lemma}
\begin{proof}
    Consider the test function $\psi=(u-v)_+ \in X_0^{s,p}(\Omega)$ with $\psi \geq 0$. Clearly, $u\geq v$ in $\operatorname{supp}\psi$. Thus, we have 
    $$H_{s,p,\Omega}(v,\psi)-H_{s,p,\Omega}(u,\psi)=\int_{\operatorname{supp}\psi}(f(x,v-f(x,u)\psi dx \geq 0.$$
    Then, following similar to the proof of \cite[Lemma 4.2]{LG2026}, we obtain $u \leq v$ a.e. in $\Omega$.
\end{proof}

Next, we show that continuous weak supersolutions to \eqref{G} are viscosity supersolutions to \eqref{G} for the case $p\geq 2$.

\medskip
\noindent {\it{\bf{Proof of Theorem} \ref{G-T1}.}}
   Let $u\in X_0^{s,p}(\Omega)$ be a continuous weak supersolution to \eqref{G}. We prove by the method of contradiction. Assume that $u$ is not a viscosity supersolution. Then, there exist $x_0\in \Omega, \ R>0$ with $B_R(x_0)\subset \Omega$ and a function $\psi\in C^2(B_R(x_0)\cap L_{s,p}^{p-1}(\mathbb{R}^N)$ with $\psi(x_0)=u(x_0)$ and $\psi \leq u$ such that
   \begin{equation*}
       \lim\limits_{\mu\rightarrow 0}\sup\limits_{x\in B_\mu(x_0)\setminus \{x_0\}}-\Delta_p \psi(x)+(-\Delta_p)^s \psi(x_0) < f(x_0,u(x_0),\nabla \psi(x_0), D_s^p \psi(x_0)).
   \end{equation*}
   Then, there exists $r\in (0,R)$ such that
   \begin{equation*}
       \sup\limits_{x\in B_{r}(x_0)\setminus \{x_0\}}-\Delta_p \psi(x)+(-\Delta_p)^s \psi(x_0) < f(x_0,u(x_0),\nabla \psi(x_0), D_s^p \psi(x_0)).
   \end{equation*}
   Observe that $\overline{B_r(x_0)}$ is a compact subset of $\mathbb{R}^N$. Thus, using the uniform continuity of the map $x \mapsto f(x,u(x), \nabla \psi(x), D_s^p\psi(x))$ in the compact set $\overline{B_r(x_0)}$ and from \cite[Lemma 3.8]{KKL19}, there exists $\rho>0$ such that 
   \begin{equation}\label{G-T1-3}
       \mathfrak{L}_{s,p}\psi(x) < f(x,u(x),\nabla \psi(x), D_s^p \psi(x))-\rho \text{ for all } x\in \overline{B_{r}(x_0)}.
   \end{equation}
   Since $p\geq 2$, {it is easy to see that for a function $v\in C_c^2(B_r(x_0))$, the convergence $\Delta_p (\psi+\alpha v) \rightarrow \Delta_p \psi$ as $\alpha \rightarrow 0$ is uniform in $B_{r}(x_0)$.} Thus, using \cite[Lemma 3.9]{KKL19}, we also get $r' \in (0,r), \alpha'>0$ and $\phi \in C_c^2(B_{\frac{r'}{2}}(x_0))$ with $\phi(x_0)=1, \ 0\leq \phi \leq 1$ such that
   \begin{equation}\label{G-T1-4}
   \left|\mathfrak{L}_{s,p}\psi(x)-\mathfrak{L}_{s,p}\psi_\alpha(x)\right|<\frac{\rho}{4},  \text{ for all }x \in B_{r'}(x_0), \alpha\in[0,\alpha'),
    \end{equation}
    where $\psi_\alpha=\psi+\alpha\phi$. Now, by the Lipschitz continuity of $f$ with respect to $\eta$ and $\zeta$, there exists $L>0$ such that
    \begin{align}\label{G-T1-1}
      |f(x,t,\eta_1,\zeta_1)-f(x,t,\eta_2,\zeta_2)|<L(|\eta_1-\eta_2|+|\zeta_1-\zeta_2|), \text{ for all } \eta_1,\eta_2\in \mathbb{R}^N, \ \zeta_1,\zeta_2 \in \mathbb{R}.
   \end{align}
   Choose $0<\alpha''<\alpha'$ sufficiently small such that
    \begin{equation}\label{G-T1-5}
        \sup\limits_{B_{2r'}(x_0)}\left(\alpha|\nabla\phi(x)|+|D_s^p \psi(x)-D_s^p \psi_\alpha(x)|\right)<\frac{\rho}{2L},
    \end{equation}
    for all $0\leq \alpha\leq\alpha''$. Combining \eqref{G-T1-3}--\eqref{G-T1-5}, we deduce
    \begin{align}\label{G-T1-6}
        \mathfrak{L}_{s,p}\psi_\alpha(x)&< \mathfrak{L}_{s,p}\psi(x)+\frac{\rho}{4}\nonumber\\
        &< f(x,u(x),\nabla \psi(x), D_s^p \psi(x))-\frac{3\rho}{4} \nonumber\\
        &\leq f(x, u(x),\nabla \psi_\alpha(x), D_s^p \psi_\alpha(x))-\frac{\rho}{4} \nonumber\\
        &<f(x, u(x),\nabla \psi_\alpha(x), D_s^p \psi_\alpha(x)), \text{ for } x\in \overline{B_{r'}(x_0)}, \  0\leq \alpha\leq\alpha''.
    \end{align}
    Multiplying \eqref{G-T1-6} by a function $v\in C_c^\infty(B_r(x_0))$ and integrating by parts, we deduce that $\psi_\alpha$ is a weak subsolution to the problem
    \begin{equation}\label{G'}
        \mathfrak{L}_{s,p} v(x)=\hat{f}(x,\nabla v(x), D_s^p v(x)), \  x\in B_{r'}(x_0),
    \end{equation}
    where $\hat{f}(x,\eta,\zeta):=f(x,u(x),\eta,\zeta)$. Clearly, $u$ is a weak supersolution to \eqref{G'} and $\psi_\alpha \leq u$ in $\mathbb{R}^N \setminus B_{r'}(x_0)$. From the continuity of $u$ and $\psi_\alpha$ and since the comparison principle holds, we deduce $ \psi_\alpha \leq u$ in $B_{r'}(x_0)$. However, this contradicts the fact that $\psi_\alpha(x_0)>u(x_0)$. Hence $u$ is a viscosity supersolution to \eqref{G}.
\hfill\qedsymbol{} \medskip

Note that the above proof does not hold for the case $1<p<2$, due to the lack of continuity of the map $x \mapsto \Delta_p \psi(x)$ when $\nabla \psi$ vanishes, even for $C^2$ functions $\psi$. For $1<p<2$, we establish that the weak supersolutions are viscosity supersolutions as follows.

\medskip
\noindent {\it{\bf{Proof of Theorem} \ref{G-T3}.}}
    Consider a continuous weak supersolution $u$ of \eqref{G} in $\Omega$. If possible, assume that $u$ is not a viscosity supersolution to \eqref{G}. Then, by Definition \ref{G-VS} and Remark \ref{G-VS-R}, there exist $B_r(x_0)\subset \Omega$ and $\psi\in C^2(B_r(x_0))\cap L_{s,p}^{p-1}(\mathbb{R}^N)$ with $\psi(x_0)=u(x_0)$, $\psi \leq u$, $\psi<u$ in $B_r(x_0)\setminus \{x_0\}$ and satisfying one of the conditions $(i)$ or $(ii)$ in Definition \ref{G-VS} such that
    \begin{equation*}
       \lim\limits_{\mu\rightarrow 0}\sup\limits_{x\in B_\mu(x_0)\setminus \{x_0\}}-\Delta_p \psi(x)+(-\Delta_p)^s \psi(x_0) < f(x_0,u(x_0),\nabla \psi(x_0), D_s^p \psi(x_0)).
   \end{equation*}
   Proceeding similar to the proof of Theorem \ref{G-T1}, we obtain $r'\in (0,\min\{\frac{r}{2},1\})$ such that $\nabla \psi (x) \neq 0$ for $x\in B_{r'}(x_0)\setminus \{x_0\}$ and
   \begin{equation}\label{G-T3-1}
       \mathfrak{L}_{s,p}\psi(x) < f(x,u(x),\nabla \psi(x), D_s^p \psi(x))-\rho \text{ for all } x\in \overline{B_{r'}(x_0)}\setminus \{x_0\}.
   \end{equation}
   Now, define the function $\phi \in C^2(B_r(x_0))$ by $\phi:=\psi-cw$, where $w\in C^\infty(\mathbb{R}^N)$ with $w\equiv 0$ in $B_{\frac{r'}{2}}(x_0)$, $w\equiv 1$ in $\mathbb{R}^N \setminus B_{r'}(x_0)$, $0\leq w \leq 1$ and the constant $c>0$ is to be chosen later. Note that $u-\phi \geq c$ in $\mathbb{R}^N\setminus B_{r'}(x_0)$. Thus, from the compactness of $\overline{B_{r'}(x_0)}\setminus B_{\frac{r'}{4}}(x_0)$ and the fact that $\phi\leq \psi <u$ in $B_{r'}(x_0)$, we have
   \begin{equation}\label{G-T3-2}
       \inf\limits_{\mathbb{R}^N \setminus B_{\frac{r'}{4}}(x_0)}(u-\phi)= \min\left\{\inf\limits_{\mathbb{R}^N\setminus B_{r'}(x_0)}(u-\phi),\inf\limits_{B_{r'}(x_0)\setminus B_{\frac{r'}{4}}(x_0)}(u-\phi)\right\}:=m>0.
   \end{equation}
   By the uniform continuity of $f:=f(x,t,\eta,\zeta)$ with respect to $\zeta$, there exists $\delta>0$ such that
   \begin{equation}\label{G-T3-10}
       |f(x,t,\eta,\zeta_1)-f(x,t,\eta,\zeta_2)|<\frac{\rho}{4} \text{ given }|\zeta_1-\zeta_2|<\delta.
   \end{equation}
   \textbf{Claim:} There exist $c_1,c_2>0$ such that
   \begin{align}
       |\mathfrak{L}_{s,p}\psi(x)-\mathfrak{L}_{s,p}\phi(x)|&<\frac{\rho}{4}, \ x\in B_{\frac{r'}{4}}(x_0), \ 0<c\leq c_1 \label{G-T3-3} \text{ and }\\
       |D_s^p\psi(x)-D_s^p\phi(x)|&<\delta, \ x\in B_{\frac{r'}{4}}(x_0), \ 0<c\leq c_2 \label{G-T3-4}.
   \end{align}
   We first prove \eqref{G-T3-3}. Let $x\in  B_{\frac{r'}{4}}(x_0)$ and $y\in \mathbb{R}^N \setminus B_{\frac{r'}{2}}(x_0)$. It is easy to see that there exists $C>0$ such that
   \begin{align}\label{G-T3-5}
       |x-y|\geq |y-x_0|-|x-x_0|\geq \frac{r'}{8}\left(\frac{8|y-x_0|}{r'}-2\right)\geq C(|y-x_0|+1)\geq C|y-x_0|.
   \end{align}
   We have the algebraic inequality
   \begin{equation}\label{G-T3-6}
       |h(t_1-t_2)-h(t_1)|\leq C|t_2|(|t_1|+|t_2|)^{p-2} \text{ for }a,b\in \mathbb{R}, \text{ and } 1<p<\infty.
   \end{equation}
   Therefore for $1<p<2$, taking $t_1=\psi(x)-\psi(y)$ and $t_2=cw(y)$ in \eqref{G-T3-6}, we get
   \begin{align}\label{G-T3-7}
       |h(\phi(x)-\phi(y))-h(\psi(x)-\psi(y))|&=|h(\psi(x)-\psi(y)-cw(y))-h(\psi(x)-\psi(y))| \nonumber\\
       &\leq Cc|w(y)|\left(|\psi(x)-\psi(y)|+c|w(y)|\right)^{p-2}\nonumber\\
       &\leq Cc^{p-1}.
   \end{align}
   Observe that the last step in \eqref{G-T3-7} is obtained using the conditions $1<p<2$ and $|w(y)|\leq 1$. We have $\phi \equiv \psi$ in $B_{\frac{r'}{2}}(x_0)$. Thus, applying \eqref{G-kernel}, \eqref{G-T3-5}, \eqref{G-T3-7} and using the fact that the constant function $1\in L_{s,p}^{p-1}(\mathbb{R}^N)$, we get
   \begin{align*}
       |\mathfrak{L}_{s,p}\psi(x)-\mathfrak{L}_{s,p}\phi(x)|&\leq \int_{\mathbb{R}^N \setminus B_{\frac{r'}{2}}(x_0)}\left|h(\phi(x)-\phi(y))-h(\psi(x)-\psi(y))\right|K_{s,p}(x,y)dy \\
       &\leq C \int_{\mathbb{R}^N \setminus B_{\frac{r'}{2}}(x_0)}\frac{\left|h(\phi(x)-\phi(y))-h(\psi(x)-\psi(y))\right|}{|y-x_0|^{N+sp}}dy\\
       &\leq Cc^{p-1}, \ x\in B_{\frac{r'}{4}}(x_0).
   \end{align*}
   Thus, we can choose $c_1>0$ sufficiently small such that \eqref{G-T3-3} holds. Next, we establish \eqref{G-T3-4}. Again, let $x\in x\in B_{\frac{r'}{4}}(x_0)$ and $y\in \mathbb{R}^N \setminus B_{\frac{r'}{2}}(x_0)$. We have the algebraic inequality
   \begin{equation}\label{G-T3-8}
       |t_1^p-t_2^p|\leq C |t_1-t_2|\max\{t_1,t_2\}^{p-1}\leq C|t_1-t_2|(|t_1|^{p-1}+|t_2|^{p-1}),\ a,b>0, \ 1<p<\infty.
   \end{equation}
   Taking $t_1=|\psi(x)-\psi(y)|$ and $t_2=|\psi(x)-\phi(y)|$ in \eqref{G-T3-8} and since $\psi$ is bounded in $B_r(x_0)$, we deduce
   \begin{align}\label{G-T3-9}
       ||\phi(x)-\phi(y)|^p-|\psi(x)-\psi(y)|^p|&\leq Cc|w(y)|\max\left\{|\phi(x)-\phi(y)|,|\psi(x)-\psi(y)|\right\}^{p-1}\nonumber\\
       &\leq Cc\left(|\phi(x)-\phi(y)|^{p-1}+|\psi(x)-\psi(y)|^{p-1}\right)\nonumber\\
       &\leq Cc\left(|\psi(x)|^{p-1}+|\psi(y)|^{p-1}+c^{p-1}|w(y)|^{p-1}\right)\nonumber\\
       &\leq Cc\left(1+|\psi(y)|^{p-1}+c^{p-1}|w(y)|^{p-1}\right).
   \end{align}
   Note that since $w$ is a bounded function, $w\in L_{s,p}^{p-1}(\mathbb{R}^N)$. Then, using \eqref{G-kernel}, \eqref{G-T3-5}, \eqref{G-T3-9} and the fact that $w, \psi \in L_{s,p}^{p-1}(\mathbb{R}^N)$, we deduce that
   \begin{align*}
       |D_s^p\psi(x)-D_s^p\phi(x)|&\leq \int_{\mathbb{R}^N \setminus B_{\frac{r'}{2}}(x_0)}\left|\phi(x)-\phi(y)|^p-|\psi(x)-\psi(y)|^p\right|K_{s,p}(x,y)dy \\
       &\leq Cc \int_{\mathbb{R}^N \setminus B_{\frac{r'}{2}}(x_0)} \frac{1+|\psi(y)|^{p-1}+c^{p-1}|w(y)|^{p-1}}{|y-x_0|}dy\nonumber\\
       &\leq Cc(1+c^{p-1}).
   \end{align*}
   Thus, it is possible to find $c_2>0$ such that \eqref{G-T3-4} holds. Now, we choose $c=\min\{c_1,c_2\}>0$. We have $ \phi(x)= \psi(x)$ for $x\in B_{\frac{r'}{2}}(x_0)$. Thus, \eqref{G-T3-1} and \eqref{G-T3-10}--\eqref{G-T3-4} gives
   \begin{align}\label{G-T3-11}
       \mathfrak{L}_{s,p}\phi(x)&<\mathfrak{L}_{s,p}\psi(x)+ \frac{\rho}{4}\nonumber\\
       &< f(x,\psi(x),\nabla \psi(x), D_s^p \psi(x))-\frac{3\rho}{4} \nonumber\\
       &<f(x,\phi(x),\nabla \phi(x), D_s^p \phi(x))-\frac{\rho}{2} \nonumber\\
       &<f(x,\phi(x),\nabla \phi(x), D_s^p \phi(x)), \ x\in B_{\frac{r'}{4}}(x_0)\setminus \{x_0\}.
   \end{align}
  Define $\phi_1:=\phi+m$. Clearly, $\phi_1\leq u$ in $\mathbb{R}^N \setminus  B_{\frac{r'}{4}}(x_0)$ by \eqref{G-T3-2}. From \eqref{G-T3-11} and the non-decreasing property of $f:=f(x,t,\eta,\zeta)$ with respect to $t$, we get
  \begin{align*}
      \mathfrak{L}_{s,p}\phi_1(x)&= \mathfrak{L}_{s,p}\phi(x)\\
      &<f(x,\phi(x),\nabla \phi(x), D_s^p \phi(x))\\
      &\leq f(x,\phi_1(x),\nabla \phi_1(x), D_s^p \phi_1(x)), \ x\in B_{\frac{r'}{4}}(x_0)\setminus \{x_0\}.
  \end{align*}
  Thus $\phi_1$ is a weak subsolution to \eqref{G} in $B_{\frac{r'}{4}}(x_0)$. Clearly $u$ is a weak supersolution to \eqref{G} in $B_{\frac{r'}{4}}(x_0)$. Then, by the comparison principle, we have $\phi_1 \leq u$ in $B_{\frac{r'}{4}}(x_0)$. This contradicts the fact that $\phi_1(x_0)=u(x_0)+m>u(x_0)$. Therefore, we deduce that $u$ is a viscosity supersolution to \eqref{G}. This completes the proof.
\hfill\qedsymbol{} \medskip

Before moving on to the next section, we have the following remark on Theorem \ref{G-T3}.
\begin{remark}\label{G-T3-R}
    When $1<p<2$, Theorem \ref{G-T3} can be applied to see that every continuous weak supersolution to \eqref{G} where $f:=0$ is a viscosity supersolution to \eqref{G}, as the comparison principle holds in this case. Theorem \ref{G-T3} can also be applied when $f:=f(x)$, since the comparison principle in this case is guarenteed by Lemma \ref{G-L1}.
\end{remark}

\section{Viscosity solutions are Weak solutions}\label{G-s4}
\noindent In this section, our aim is to prove that the viscosity solutions to \eqref{G} are weak solutions to the same. We begin by proving that the infimal convolutions of the viscosity supersolutions to \eqref{G} are viscosity supersolutions to a problem generated from \eqref{G}.
\begin{lemma}\label{G-L2}
    Assume that $f:=f(x,t,\eta, \zeta)$ is a continuous function, which is non-increasing in $t$ and $u$ is a viscosity supersolution to \eqref{G}. For each $\epsilon>0$ and $1<p<\infty$, the infimal convolution $u_\epsilon$ of $u$ is a viscosity supersolution to the problem
    \begin{equation}\label{G-L2-1}
        \mathfrak{L}_{s,p}v=f^\epsilon(x,v,\nabla v, D_s^p v)
    \end{equation}
    in $\Omega_\epsilon$, where the function $f^\epsilon$ is defined by
    \begin{equation}\label{G-L2-2}
       f^\epsilon(x,t,\eta,\zeta):=\inf\limits_{y\in B_{r(\epsilon)}(x)}f(y,t,\eta,\zeta),
    \end{equation}
    with $r(\epsilon)$ given by Lemma \ref{G-def1-R}.
\end{lemma}
\begin{proof}
    For each $\epsilon>0$ and $b\in B_{r(\epsilon)}(0)$, let us define the function $w_b$ by
    $$w_b(x)=u(x+b)+\frac{|b|^q}{q\epsilon^{q-1}}, \ x\in \mathbb{R}^N.$$
    Assume that $w_b$ is a viscosity solution to \eqref{G-L2-1} in $\Omega_\epsilon$. Observe that by \cite[Lemma 3.1(v)]{BM21}, for any $x_0\in \Omega_\epsilon$ and $r'>0$ with $B_{r'}(x_0)\subset \Omega_\epsilon$, there exists $b\in B_{r(\epsilon)}(0)$ such that $u_\epsilon(x_0)=w_b(x_0)$. Clearly, $u_\epsilon \leq w_b$ by the definition of $w_b$. Consider any $\psi\in C^2(B_{r'}(x_0))\cap L_{s,p}^{p-1}(\mathbb{R}^N)$ with $\psi\leq u_\epsilon, \ \psi(x_0)=u_\epsilon(x_0)$ and such that if $p<2$ and $\nabla \psi(x_0)=0$, then $x_0$ is an isolated point of $\psi$. Furthermore, there exists $\beta>\frac{sp}{p-1}$ such that $\psi\in C_\beta^2(B_r(x_0))$ if $p \leq \frac{2}{2-s}$. Then, $\psi  \leq w_b$. Since $w_b$ is a viscosity supersolution, we obtain
    $$\lim\limits_{r\rightarrow 0}\sup\limits_{x\in B_r(x_0)\setminus \{x_0\}}-\Delta_p \psi(x)+(-\Delta_p)^s \psi(x_0) \geq f_\epsilon(x_0),\psi(x_0),\nabla \psi(x_0), D_s^p \psi(x_0)).$$
    Thus, $u_\epsilon$ is a viscosity supersolution to \eqref{G-L2-1} in $\Omega_\epsilon$. Hence, it suffices to show that for every $b\in B_{r(\epsilon)}(0)$, $w_b$ is a viscosity supersolution to \eqref{G-L2-1} in $\Omega_\epsilon$. Again, consider any point $x_0\in \Omega_\epsilon$ and $r'>0$ with $B_{r'}(x_0)\subset \Omega_\epsilon$. Let $\psi\in C^2(B_{r'}(x_0))\cap L_{s,p}^{p-1}(\mathbb{R}^N)$ with $\psi\leq w_b, \ \psi(x_0)=w_b(x_0)$ and such that if $\nabla \psi$ has a zero at $x_0$, then $x_0$ is an isolated zero for $p<2$ and there exists $\beta>\frac{sp}{p-1}$ such that $\psi \in C_\beta^2(B_r(x_0))$ if $p\leq \frac{2}{2-s}$. Let us define $\phi$ by 
    $$\phi(y)=\psi(y-b)-\frac{|b|^q}{q\epsilon^{q-1}}, \ y\in \mathbb{R}^N.$$
     Then, $\phi\in C^2(B_{r'}(x_0+b))$. Also, note that $u(y)=w_b(y-b)-\frac{|b|^q}{q\epsilon^{q-1}}$ for all $y\in \mathbb{R}^N$. Hence, $\phi \leq u$ and $\phi(x_0+b)=u(x_0+b)$. Since $u$ is a viscosity supersolution to \eqref{G} in $\Omega$, we get
     \begin{align*}
         \lim\limits_{r\rightarrow 0}&\sup\limits_{x\in B_r(x_0)\setminus \{x_0\}}-\Delta_p \psi(x)+(-\Delta_p)^s \psi(x_0)\\
         &=\lim\limits_{r\rightarrow 0}\sup\limits_{x\in B_r(x_0)\setminus \{x_0\}}-\Delta_p \phi(x+b)+(-\Delta_p)^s \phi(x_0+b)\\
         &\geq f(x_0+b,\phi(x_0+b), \nabla \phi(x_0+b), D_s^p \phi(x_0+b))\\
         &=f\left(x_0+b,\psi(x_0)-\frac{|b|^q}{q\epsilon^{q-1}}, \nabla \psi(x_0), D_s^p \psi(x_0)\right)\\
         &\geq f^\epsilon(x_0,\psi(x_0), \nabla \psi(x_0), D_s^p \psi(x_0)),
     \end{align*}
     where the last step is obtained using the non-increasing property of $f:=f(x,t,\eta,\zeta)$ with respect to $t$ and the definition of $f^\epsilon$. Thus, $w_b$ is a viscosity supersolution to \eqref{G-L2-1} in $\Omega_\epsilon$. This completes the proof.
\end{proof}
Next, we show that $u_\epsilon$ satisfies \eqref{G-L2-1} in the weak sense for $p \geq 2$. The proof for the case $1<p<2$ has not been obtained due to the difficulty in choosing a suitable test function for viscosity solutions in this range. 
\begin{lemma}\label{G-L3}
    Let $2 \leq p <\infty$ and $u$ be a viscosity supersolution to \eqref{G}. Then, for any $v\in C_c^\infty(\Omega_\epsilon)$ with $v \geq 0$,
    we have
    $$H_{s,p,\Omega_\epsilon}(u_\epsilon,v)\geq \int_{\Omega_\epsilon}f^\epsilon(x,u_\epsilon, \nabla u_\epsilon, D_s^p u_\epsilon )vdx,$$
    where $f^\epsilon$ is given by \eqref{G-L2-2}.
\end{lemma}
\begin{proof}
    Using \cite[Lemma 3.1(iii)]{BM21} and Lusin's theorem, we deduce that for almost every $x_0\in \Omega_\epsilon$, there exists $r>0$ such that $u_\epsilon\in C^2(B_r(x_0))$.  Observe that $u_\epsilon$ is a viscosity supersolution to \eqref{G-L2-1} by Lemma \ref{G-L2}. Therefore, we get $\mathfrak{L}_{s,p}u_\epsilon\geq f^\epsilon(x,u_\epsilon(x), \nabla u_\epsilon(x), D_s^p u_\epsilon(x))$ a.e. in $\Omega_\epsilon$. Then, for every $v\in C_c^\infty(\Omega_\epsilon)$ with $v \geq 0$, we get
    \begin{equation}\label{G-L3-1}
        \int_{\Omega_\epsilon}\mathfrak{L}_{s,p}u_\epsilon v dx \geq \int_{\Omega_\epsilon}f^\epsilon(x,u_\epsilon, \nabla u_\epsilon, D_s^p u_\epsilon) v dx.
    \end{equation}
    Following the proof of (3.16) and (3.21) in \cite[Lemma 3.3]{LG2026}, we deduce
    \begin{equation}\label{G-L3-2}
        H_{s,p,\Omega_\epsilon}(u_\epsilon,v)\geq \int_{\Omega_\epsilon}\mathfrak{L}_{s,p}u_\epsilon v dx.
    \end{equation}
    Combining \eqref{G-L3-1} and \eqref{G-L3-2}, we get the desired result.
\end{proof}
Now, we establish a limiting property of the RHS in \eqref{G-L3-1} with respect to the function $u$.  
  \begin{lemma}\label{G-L4}
      Let $1<p<\infty$, $u\in X_0^{s,p}(\Omega)\cap L_{s,p}^{p-1}(\mathbb{R}^N)$ and $f:=f(x,t,\eta,\zeta)$ be a uniformly continuous function, which is also Lipschitz continuous in $\eta$ and $\zeta$. Additionally, assume that $f$ satisfies \eqref{G-f condn} and $f^\epsilon$ is given by \eqref{G-L2-2}. Then, for any $v\in C_c^\infty(\Omega)$ with $K:=\operatorname{supp}(v)$, $v\geq 0$ and
      \begin{equation}\label{G-L4-1}
          \lim\limits_{\epsilon\rightarrow 0}\left(\int_K (|\nabla u_\epsilon-\nabla u|^p)dx+\int_K\int_{\mathbb{R}^N}|u_\epsilon(x)-u_\epsilon(y)-u(x)+u(y)|^pK_{s,p}(x,y)dydx\right)=0,
      \end{equation}
      the following limit holds.
      \begin{equation*}
          \lim\limits_{\epsilon\rightarrow 0}\int_K f^\epsilon(x,u_\epsilon,\nabla u_\epsilon, D_s^p u_\epsilon)vdx=\int_K f(x,u,\nabla u, D_s^p u)vdx.
      \end{equation*}
  \end{lemma}
  \begin{proof}
      Observe that
      \begin{align}\label{G-L4-2}
          \Bigg|\int_K \big(f^\epsilon(x,u_\epsilon,\nabla u_\epsilon&, D_s^p u_\epsilon)-f(x,u,\nabla u, D_s^p u)\big)vdx\Bigg|\nonumber\\
          &\leq \int_K |f^\epsilon(x,u_\epsilon,\nabla u_\epsilon, D_s^p u_\epsilon)-f(x,u_\epsilon,\nabla u_\epsilon, D_s^p u_\epsilon)|vdx \nonumber\\
          &\quad +\int_K |f(x,u_\epsilon,\nabla u_\epsilon, D_s^p u_\epsilon)-f(x,u_\epsilon,\nabla u_\epsilon, D_s^p u)|vdx \nonumber\\
          &\quad +\int_K |f(x,u_\epsilon,\nabla u_\epsilon, D_s^p u)-f(x,u_\epsilon,\nabla u, D_s^p u)|vdx \nonumber\\
          &\quad +\int_K |f(x,u_\epsilon,\nabla u, D_s^p u)-f(x,u,\nabla u, D_s^p u)|vdx.
      \end{align}
      Clearly, $v$ is bounded in $\mathbb{R}^N$. Now, from the Lipschitz continuity of $f$ with respect to $\eta$, $\zeta$ and using the H\"older inequality, we get
      \begin{align}\label{G-L4-3}
          \int_K |f(x,u_\epsilon&,\nabla u_\epsilon, D_s^p u_\epsilon)-f(x,u_\epsilon,\nabla u_\epsilon, D_s^p u)|vdx \nonumber\\
          \quad &+\int_K |f(x,u_\epsilon,\nabla u_\epsilon, D_s^p u)-f(x,u_\epsilon,\nabla u, D_s^p u)|vdx \nonumber\\
          &\leq C\int_K | D_s^p u_\epsilon-D_s^p u|vdx +C\int_K |\nabla u_\epsilon-\nabla u|vdx \nonumber\\
          &\leq C\int_K\int_{\mathbb{R}^N}(|u_\epsilon(x)-u_\epsilon(y)|^p-|u(x)-u(y)|^p)K_{s,p}(x,y)dydx\nonumber\\
          & \quad +C\left(\int_K|\nabla u_\epsilon-\nabla u|^pdx\right)^\frac{1}{p}\left(\int_K|v|^\frac{p}{p-1}\right)^\frac{p-1}{p}.
      \end{align}
      Also, we have
      \begin{equation}\label{G-L4-4}
          ||t_1|^p-|t_2|^p|\leq C(|t_1-t_2|+|t_2|)^{p-1}|t_1-t_2|, ~t_1,t_2\in \mathbb{R}.
      \end{equation}
      Using \eqref{G-L4-4} with $t_1=u_\epsilon(x)-u_\epsilon(y)$ and $t_2=u(x)-u(y)$ and applying the H\"older inequality, we obtain
      \begin{align}\label{G-L4-5}
          \int_K&\int_{\mathbb{R}^N}(|u_\epsilon(x)-u_\epsilon(y)|^p-|u(x)-u(y)|^p)K_{s,p}(x,y)dydx\nonumber\\
          &\leq C\int_K\int_{\mathbb{R}^N}\left(|u_\epsilon(x)-u_\epsilon(y)-u(x)+u(y)|+|u(x)-u(y)|\right)^{p-1}\nonumber\\
          &\quad \times |u_\epsilon(x)-u_\epsilon(y)-u(x)+u(y)|K_{s,p}(x,y)dydx\nonumber\\
          &\leq C\left(\int_K\int_{\mathbb{R}^N}\left(|u_\epsilon(x)-u_\epsilon(y)-u(x)+u(y)|^p+|u(x)-u(y)|^p\right)K_{s,p}(x,y)dydx\right)^\frac{p-1}{p}\nonumber\\
          &\quad \times \left(\int_K\int_{\mathbb{R}^N}\left(|u_\epsilon(x)-u_\epsilon(y)-u(x)+u(y)|\right)^pK_{s,p}(x,y)dydx\right)^\frac{1}{p}.
      \end{align}
      Note that $u\in W^{s,p}(\Omega)$ and $v\in L^\frac{p}{p-1}(\Omega)$. Thus, using \eqref{G-L4-3}, \eqref{G-L4-5} and finally applying \eqref{G-L4-1}, we deduce
      \begin{align}\label{G-L4-6}
          \int_K |f(x,u_\epsilon&,\nabla u_\epsilon, D_s^p u_\epsilon)-f(x,u_\epsilon,\nabla u_\epsilon, D_s^p u)|vdx \nonumber\\
          \quad &+\int_K |f(x,u_\epsilon,\nabla u_\epsilon, D_s^p u)-f(x,u_\epsilon,\nabla u, D_s^p u)|vdx
          \rightarrow 0 \text{ as }\epsilon\rightarrow 0.
      \end{align}
      From the uniform continuity of $f$, for any $\alpha>0$, we get $\rho>0$ such that
      \begin{equation}\label{G-L4-7}
          |f(x,u_\epsilon(x),\nabla u_\epsilon(x), D_s^p u_\epsilon(x))-f(y,u_\epsilon(x),\nabla u_\epsilon(x), D_s^p u_\epsilon(x))|\leq \alpha, \text{ if } |x-y|<\rho.
      \end{equation}
      From the definition of $f^\epsilon$, and \eqref{G-L4-7}, we get
      \begin{equation}\label{G-L4-8}
          |f^\epsilon(x,u_\epsilon(x),\nabla u_\epsilon(x), D_s^p u_\epsilon(x))-f(x,u_\epsilon(x),\nabla u_\epsilon(x), D_s^p u_\epsilon(x))|\leq \alpha, \text{ if } r_\epsilon<\rho,x\in \Omega_\epsilon.
      \end{equation}
      Choose $\epsilon>0$ sufficiently small such that $r_\epsilon<\rho$. Then, \eqref{G-L4-8} and the boundedness of $v$ in $\mathbb{R}^N$ give
      \begin{align*}
          \int_K |f^\epsilon(x,u_\epsilon(x),\nabla u_\epsilon(x), D_s^p u_\epsilon(x))-f(x,u_\epsilon(x),\nabla u_\epsilon(x), D_s^p u_\epsilon(x))|vdx\leq C|K|\alpha.
      \end{align*}
      Since $\alpha>0$ is arbitrary and $r_\epsilon\rightarrow 0$ as $\epsilon\rightarrow 0$, we get
      \begin{align}\label{G-L4-10}
          \int_K |f^\epsilon(x,u_\epsilon(x),\nabla u_\epsilon(x), D_s^p u_\epsilon(x))-f(x,u_\epsilon(x),\nabla u_\epsilon(x), D_s^p u_\epsilon(x))|vdx\rightarrow 0 \text { as }\epsilon\rightarrow 0.
      \end{align}
      Finally, we have
      \begin{align}\label{G-L4-11}
          |f(x,u_\epsilon(x),\nabla u(x), D_s^p u(x))-f(x,u(x),\nabla u(x), D_s^p u(x))|v(x) \rightarrow 0
      \end{align}
      pointwise in $K$. By \eqref{G-L4-1} and \cite[Theorem 4.9]{B11}, there exist functions $h_1,h_2 \in L^p(K)$ such that $|\nabla u_\epsilon|\leq h_1$ and $|D_s^p u_\epsilon|^\frac{1}{p}\leq h_2$ a.e. in $K$. Also, by \eqref{G-f condn}, we have
      \begin{align}\label{G-L4-12}
          |f(x,&u_\epsilon(x),\nabla u(x), D_s^p u(x))-f(x,u(x),\nabla u(x), D_s^p u(x))|v\nonumber\\
          &\leq l_1\left(h_1(x)^{p-1}+h_2(x)^{p-1}+|\nabla u(x)|^{p-1}+|D_s^p u(x)|^{\frac{p-1}{p}}+\|g_3\|_{L^\infty(K)}\right)v\in L^1(K),
      \end{align}
      where $l_1\in [0,\infty)$ is given by
     \begin{align}\label{G-L4-13}
         l_1&=\sup\left\{|g_1(t)|+|g_2(t)|:-\|u\|_{L^\infty(\Omega)}\leq t \leq \|u\|_{L^\infty(\Omega)}\right\}\nonumber\\
         &\geq \sup\left\{|g_1(t)|+|g_2(t)|:-\|u_\epsilon\|_{L^\infty(\Omega)}\leq t \leq \|u_\epsilon\|_{L^\infty(\Omega)}\right\}.
     \end{align}
     The last inequality in \eqref{G-L4-13} is obtained by \cite[Remark 3.2]{BM21}. Using \eqref{G-L4-11}, \eqref{G-L4-12} and the dominated convergence theorem, we get
     \begin{equation}\label{G-L4-14}
         \int_K |f(x,u_\epsilon,\nabla u, D_s^p u)-f(x,u,\nabla u, D_s^p u)|vdx\rightarrow 0 \text{ as } \epsilon\rightarrow 0.
     \end{equation}
     Taking the limit as $\epsilon\rightarrow0$ in \eqref{G-L4-2} and applying \eqref{G-L4-6}, \eqref{G-L4-10} and \eqref{G-L4-14}, we get the desired result.
  \end{proof}
We finally establish the following inequality, which is crucial for the limiting argument in Theorem \ref{G-T2}.
  \begin{lemma}\label{G-L5}
      Let $1<p<\infty$ and let $f$ be continuous satisfying \eqref{G-f condn}. Let $u$ be a bounded weak supersolution to \eqref{G}. Then, for all $\psi \in C_c^\infty(\Omega)$ with $\operatorname{supp}\psi=K$ and $0 \leq \psi \leq 1$, $u$ satisifes the inequality
      \begin{align*}
           \int_K |\nabla u&|^{p}\psi(x)^p dx +\int_K \int_{\mathbb{R}^N}{|u(x)-u(y)|^p}\psi(x)^pK_{s,p}(x,y) dy dx \\
           &\leq C\Bigg(\left(M_u(g_1)^p+M_u(g_2)^p+\int_K|\nabla \psi|^p+\int_K\int_{\mathbb{R}^N} {|\psi(x)-\psi(y)|^{p}}{K_{s,p}(x,y)} dy dx\right)\nonumber\\
           &\ \ \ \times(\operatorname{osc} u)^p+(\operatorname{osc} u)\Bigg),
       \end{align*}
      where $\operatorname{osc}u=\sup\limits_{x\in \mathbb{R}^N}u-\inf\limits_{x\in \mathbb{R}^N}u$, $M_u(g_i)=\sup\{|g_i(|t|):-\|u\|_{L^\infty(\Omega)}\leq t \leq \|u\|_{L^\infty(\Omega)}\}$ and $C=C(N,p,K,g_3)>0$ is a constant.
  \end{lemma}
  \begin{proof}
      We define the function $w$ by
      $$w(x):=\begin{cases}
            \left(\sup\limits_{z\in \mathbb{R}^N} u(z)-u(x)\right) \psi(x)^p, & \ x\in \Omega, \\
            0, & \ x \in \mathbb{R}^N \setminus \Omega.
        \end{cases}$$
       Taking $w$ as a test function for the weak supersolution $u$ in \eqref{G}, we get
       \begin{align}\label{G-L5-1}
           \int_K &f(x,u,\nabla u, D_s^p u)wdx\nonumber\\
           &\leq \int_K |\nabla u|^{p-2} \nabla u \cdot \nabla w dx + \int_{\mathbb{R}^N} \int_{\mathbb{R}^N} h(u(x)-u(y))(w(x)-w(y))K_{s,p}(x,y) dy dx\nonumber\\
           &=\int_K |\nabla u|^{p-2} \nabla u \cdot \left(-\psi^p \nabla u  + p \psi^{p-1}\left(\sup\limits_{z\in\mathbb{R}^N} u(z)-u \right)\nabla \psi\right)dx\nonumber\\
           &\quad +\iint_{A(K)}h(u(x)-u(y))\left((\psi(x)^p-\psi(y)^p)\left(\sup\limits_{z\in\mathbb{R}^N} u(z)-u(y)\right)-(u(x)-u(y))\psi(x)^p\right)\nonumber\\
           &\hspace{1.5cm} \times K_{s,p}(x,y)dxdy.
       \end{align}
       Simplifying and rearranging the terms in \eqref{G-L5-1}, we get
       \begin{align}\label{G-L5-2}
           \int_K& |\nabla u|^{p}\psi(x)^p dx +\int_K \int_{\mathbb{R}^N}{|u(x)-u(y)|^p}\psi(x)^pK_{s,p}(x,y) dy dx \nonumber\\
           &\leq -\int_K f(x,u(x),\nabla u(x), D_s^p u(x))wdx+ p\underbrace{\int_\Omega \psi^{p-1} (\sup\limits_{z\in\mathbb{R}^N} u(z)-u) |\nabla u|^{p-2} \nabla u \cdot \nabla \psi  dx}_{J_1}  \nonumber \\
        &\quad + \underbrace{\iint_{A(K)}h(u(x)-u(y))(\psi(x)^p-\psi(y)^p)\left(\sup\limits_{z\in\mathbb{R}^N} u(z)-u(y)\right)K_{s,p}(x,y) dy dx}_{J_2}.
       \end{align}
       From \eqref{G-f condn}, we have
       \begin{align}\label{G-L5-5}
           \Bigg|\int_K f(x,u,\nabla u, D_s^p u)wdx\Bigg|&\leq \int_K\left( g_1(|t|)|\nabla u|^{p-1}+g_2(|t|)|D_s^p u|^{\frac{p-1}{p}}+g_3(x)\right)(\operatorname{osc} u)\psi^pdx\nonumber\\
           &\leq\int_K\left( M_u(g_1)|\nabla u|^{p-1}+M_u(g_2)|D_s^p u|^{\frac{p-1}{p}}\right)(\operatorname{osc} u)\psi^p dx\nonumber\\
           &\ \ \ \quad\quad+C(K,g_3)(\operatorname{osc} u),
       \end{align}
       where $C(K, g_3)=\|g_3\|_{L^\infty(K)}|K|$. For $\delta>0$ and $k>1$, we have the Young inequality given by
       \begin{equation}\label{G-Young}
           |ab|\leq \delta |a|^k+C(\delta,k)|b|^\frac{k}{k-1}, a,b\in \mathbb{R}.
       \end{equation}
       Observe that $0\leq \psi \leq 1$. Thus, using \eqref{G-Young} with $k=\frac{p}{p-1},a=|\nabla u|^{p-1}$ and $b=M_u(g_1+g_2)\operatorname{osc} u$, we deduce
       \begin{align}\label{G-L5-6}
           M_u(g_1)\int_K |\nabla u|^{p-1}(\operatorname{osc} u)\psi^p dx&\leq \delta \int_K |\nabla u|^p \psi^p dx+ C(\delta,p) M_u(g_1)^p (\operatorname{osc} u)^p |K|.
       \end{align}
       Similarly, we get
       \begin{align}\label{G-L5-7}
           M_u(g_2)\int_K |D_s^p u|^{\frac{p-1}{p}}(\operatorname{osc} u)\psi^p dx&\leq \delta \int_K \int_{\mathbb{R}^N}{|u(x)-u(y)|^p}\psi(x)^pK_{s,p}(x,y) dy dx\nonumber\\
           &\ \ \ \quad\quad+ C(\delta,p) M_u(g_2)^p (\operatorname{osc} u)^p |K|.
       \end{align}
       Substituting \eqref{G-L5-6} and \eqref{G-L5-7} in \eqref{G-L5-5}, we obtain
       \begin{align}\label{G-L5-8}
           \Bigg|\int_K f(x,u,&\nabla u, D_s^p u)wdx\Bigg|\nonumber\\
           &\leq \delta \Bigg(\int_K |\nabla u|^p \psi^p dx+\int_K \int_{\mathbb{R}^N}{|u(x)-u(y)|^p}\psi(x)^pK_{s,p}(x,y) dy dx\Bigg)\nonumber\\
           &\ \ \ +C(\delta,p,K)(M_u(g_1)^p+M_u(g_2)^p)(\operatorname{osc} u)^p.
       \end{align}
       For an arbitrary $\delta>0$, using \eqref{G-Young}, it is easy to see that 
       \begin{align}\label{G-L5-3}
            J_1 &\leq \delta \int_\Omega |\nabla u|^{p}\psi(x)^p dx + C(p,\delta) \int_\Omega p^p |\nabla \psi|^p(\operatorname{osc} u)^p dx.
       \end{align}
       Estimating $J_2$ similar to (3.14) in \cite[Lemma 3.2]{LG2026}, we get
       \begin{align}\label{G-L5-4}
           J_2        &\leq  C(p) \delta \iint_{W(\Sigma)} {|u(x)-u(y)|^{p}(|\psi(x)|^p +|\psi(y)|^p)}{K_{s,p}(x,y)} dy dx \nonumber\\
        &\ \ \ + C(p, \delta) (\operatorname{osc} u)^p \int_K\int_{\mathbb{R}^N} {|\psi(x)-\psi(y)|^{p}}{K_{s,p}(x,y)} dy dx.
       \end{align}
       Combining \eqref{G-L5-8}--\eqref{G-L5-4}, substituting in \eqref{G-L5-2}, and finally choosing $\delta>0$ sufficiently small, the required result is obtained.
  \end{proof}
  We conclude with the proof of Theorem \ref{G-T2}.  
  
  \medskip
  \noindent {\it{\bf{Proof of Theorem} \ref{G-T2}.}}
      Let $v\in C_c^\infty(\Omega)$ with $v\geq 0$ and $K=\operatorname{supp}v$. Choose $\epsilon'>0$ sufficiently small such that $K \subset \Omega_\epsilon$ for all $0<\epsilon\leq \epsilon'$. For the remainder of the proof, we consider only $\epsilon\in (0,\epsilon')$. By Lemma \ref{G-L3}, we obtain that
      \begin{equation}\label{G-T2-1}
          H_{s,p,\Omega_\epsilon}(u_\epsilon,v)\geq \int_{\Omega_\epsilon}f^\epsilon(x,u_\epsilon, \nabla u_\epsilon D_s^p u_\epsilon)vdx.
      \end{equation}
      Consider two compact sets $K_1, \ K_2$ and {a Lipschitz open set $\mathcal{D}$ such that $K\subset \mathcal{D} \subset K_1\subset K_2\subset \Omega_\epsilon$.} Now, consider a function $\psi\in C_c^\infty(\Omega_\epsilon)$ with $0\leq\psi \leq 1$, $\psi \equiv 1$ in $K_1$ and $\operatorname{supp}\psi=K_2$. Applying Lemma \ref{G-L5}, we get
      \begin{align}\label{G-T2-2}
          \int_{K_1} |\nabla& u_\epsilon|^p  dx +  \int_{K_1} \int_{\mathbb{R}^N}{|u_\epsilon(x)-u_\epsilon(y)|^p}{K_{s,p}(x,y)} dy dx  \nonumber\\
          &\leq  \int_{K_2} |\nabla u_\epsilon|^p\psi(x)^p dx +  \int_{K} \int_{\mathbb{R}^N}|u_\epsilon(x)-u_\epsilon(y)|^p\psi(x)^pK_{s,p}(x,y) dy dx  \nonumber\\
          &\leq C\Bigg(\left(M_{u_\epsilon}(g_1)^p+M_{u_\epsilon}(g_2)^p+\int_K|\nabla \psi|^p+\int_K\int_{\mathbb{R}^N} {|\psi(x)-\psi(y)|^{p}}{K_{s,p}(x,y)} dy dx\right)\nonumber\\
           &\ \ \ \times(\operatorname{osc} u_\epsilon)^p+(\operatorname{osc} u_\epsilon)\Bigg).
      \end{align}
      Observe that by \cite[Remark 3.2]{BM21}, we have $$|M_{u_\epsilon}(g_1)|+|M_{u_\epsilon}(g_2|)|+(\operatorname{osc} u_\epsilon)<|M_{u}(g_1)|+|M_{u}(g_2|)|+(\sup u -\inf u_{\epsilon'})<\infty$$ for all $0<\epsilon<\epsilon'$. Therefore, \eqref{G-T2-2} gives
      \begin{equation}\label{G-T2-3}
          \int_{K_1} |\nabla u_\epsilon|^p  dx +  \int_{K_1} \int_{\mathbb{R}^N}{|u_\epsilon(x)-u_\epsilon(y)|^p}{K_{s,p}(x,y)} dy dx \leq C.
      \end{equation}
      Proceeding similar to the proof of \cite[Theorem 1.1]{LG2026}, as $\epsilon\rightarrow 0$, we deduce that up to a subsequence, $\nabla u_\epsilon\rightharpoonup u$ in $L^p(K_1)$ and $u_\epsilon \rightharpoonup u$ in $W^{1,p}(K_1)$. Clearly, $u_\epsilon \rightharpoonup u$ {also in $W^{1,p}(\mathcal{D})$. By \cite[Corollary 4.34, Theorem 4.54]{DD12} and the fact that the embedding $C^{0,\alpha}(\mathcal{D}) \hookrightarrow L^r(\mathcal{D})$ is continuous and compact for $1\leq r<\infty$,} we get $u_\epsilon\rightarrow u$ up to a subsequence in $L^p(K_1)$. Following the steps of proof of \cite[Theorem 1.1]{LG2026}, we also get that as $\epsilon\rightarrow 0$,
      \begin{equation}\label{G-T2-4}
          \frac{u_\epsilon(x)-u_\epsilon(y)}{|x-y|^{\frac{N}{p}+s}}\rightharpoonup \frac{u(x)-u(y)}{|x-y|^{\frac{N}{p}+s}} \text{ in } L^p(K\times \mathbb{R}^N), \text{ and also in }  L^p(\mathbb{R}^N\times K).
      \end{equation}
      Finally, we also get $|\nabla u_\epsilon|^{p-2}\nabla u_\epsilon \rightharpoonup|\nabla u|^{p-2}\nabla u$ in $L^\frac{p}{p-1}(K_1)$ and
      \begin{equation}\label{G-T2-5}
          \frac{h(u_\epsilon(x)-u_\epsilon(y))}{|x-y|^{\frac{(N+ps)(p-1)}{p}}}\rightharpoonup \frac{h(u(x)-u(y))}{|x-y|^{\frac{(N+ps)(p-1)}{p}}},
      \end{equation}
      in both $L^{\frac{p}{p-1}}(K\times \mathbb{R}^N)$ and $L^{\frac{p}{p-1}}(\mathbb{R}^N\times K)$. Therefore, we obtain 
      \begin{equation}\label{G-T2-6}
          H_{s,p,\Omega_\epsilon}(u_\epsilon,v) \rightarrow H_{s,p,\Omega}(u,v) \text{ as } \epsilon\rightarrow 0.
      \end{equation}
      Now, consider a function $\phi \in C_c^\infty(\Omega_\epsilon)$ with $\phi \equiv 1$ in $K$ and $\operatorname{supp}\phi=K_1$. Take $v=\phi(u-u_\epsilon)$. Since $u-u_\epsilon\rightarrow 0$ in $L^p(K_1)$, using the H\"older inequality, it is easy to see that
      \begin{align}\label{G-T2-8}
          \lim\limits_{\epsilon\rightarrow 0}\int_{K_1}|\nabla u_\epsilon|^{p-2}\nabla u_\epsilon \nabla v&=\lim\limits_{\epsilon\rightarrow 0}\Bigg(\int_{K_1}|\nabla u_\epsilon|^{p-2}\nabla u_\epsilon \cdot \nabla \phi (u-u_\epsilon)dx\nonumber\\
          & \quad +\int_{K_1}|\nabla u_\epsilon|^{p-2}\nabla u_\epsilon \cdot (\nabla u-\nabla u_\epsilon) \phi dx\Bigg)\nonumber\\
          &=\lim\limits_{\epsilon\rightarrow 0}\int_{K_1}\left(|\nabla u_\epsilon|^{p-2}\nabla u_\epsilon-|\nabla u|^{p-2}\nabla u\right) \cdot (\nabla u-\nabla u_\epsilon) \phi dx\nonumber\\
         &:=-\lim\limits_{\epsilon \rightarrow 0}J'_1.
      \end{align}
      Thus, taking $\epsilon\rightarrow 0$ in \eqref{G-T2-1} with $v=\phi(u-u_\epsilon)$ and applying \eqref{G-T2-8} gives 
      \begin{align}\label{G-T2-7}
           \lim\limits_{\epsilon\rightarrow 0}\Bigg(&\int_{K_1}|\nabla u_\epsilon|^{p-2}\nabla u_\epsilon \cdot (\nabla u-\nabla u_\epsilon) \phi dx+\iint_{A(K_1)} h(u_\epsilon(x)-u_\epsilon(y))(v(x)-v(y))K_{s,p}(x,y)dydx\nonumber\\
           &\geq \lim\limits_{\epsilon\rightarrow 0}\int_{K_1}f^\epsilon(x,u_\epsilon, \nabla u_\epsilon, D_s^p u_\epsilon)vdx.
      \end{align}
      {Estimating the LHS of \eqref{G-T2-7} following the steps similar to pages 2006--2007 of \cite[Theorem 1.1]{BM21}}, we get
      \begin{align}\label{G-T2-9}
          \lim\limits_{\epsilon\rightarrow 0}&\iint_{A(K_1)} h(u_\epsilon(x)-u_\epsilon(y))(v(x)-v(y))K_{s,p}(x,y)dydx\nonumber\\
          &=-\lim\limits_{\epsilon\rightarrow 0}\int_{K_1}\int_{\mathbb{R}^N}(h(u(x)-u(y))-h(u_\epsilon(x)-u_\epsilon(y)))\nonumber\\
          &\quad \ \ \ \times(u(x)-u(y)-(u_\epsilon(x)-u_\epsilon(y)))\phi(x)K_{s,p}(x,y)dydx\nonumber\\
          &:=-\lim\limits_{\epsilon\rightarrow 0}J'_2.         
      \end{align}
      Using the algebraic inequality
      $$|a-b|^p\leq \frac{1}{2^{p-1}}|h(a)-h(b)|(a-b), a,b\in \mathbb{R}, p\geq 2,$$
      we deduce
      \begin{align}
      0&\leq \int_K|\nabla u-\nabla u_\epsilon|^pdx \leq CJ'_1,\text{ and }\label{G-T2-12}\\
          0&\leq \int_{K}\int_{\mathbb{R}^N}\Big|u(x)-u(y)-(u_\epsilon(x)-u_\epsilon(y))\Big|^pK_{s,p}(x,y)dydx\leq CJ'_2, \label{G-T2-10}
      \end{align}
      for a constant $C>0$ and for $p \geq 2$. Substituting \eqref{G-T2-10} in \eqref{G-T2-9}, we have
      \begin{align}\label{G-T2-11}
          \lim\limits_{\epsilon\rightarrow 0}&\iint_{A(K_1)} h(u_\epsilon(x)-u_\epsilon(y))(v(x)-v(y))K_{s,p}(x,y)dydx\nonumber\\
          &\leq -\lim\limits_{\epsilon\rightarrow 0}\int_{K}\int_{\mathbb{R}^N}\Big|u(x)-u(y)-(u_\epsilon(x)-u_\epsilon(y))\Big|^pK_{s,p}(x,y)dydx.
      \end{align}
      Applying \eqref{G-T2-12} and \eqref{G-T2-11} in \eqref{G-T2-7}, we arrive at
      \begin{align}\label{G-T2-13}
          0&\leq \lim\limits_{\epsilon\rightarrow 0}\Bigg(\int_K|\nabla u_\epsilon-\nabla u|^pdx+\int_{K}\int_{\mathbb{R}^N}\Big|u(x)-u(y)-(u_\epsilon(x)-u_\epsilon(y))\Big|^pK_{s,p}(x,y)dydx\Bigg)\nonumber\\
          &\leq -\int_{K_1}\int_{K_1}f^\epsilon(x,u_\epsilon(x), \nabla u_\epsilon(x), D_s^p u_\epsilon(x))(u-u_\epsilon)\phi dx.
      \end{align}
      Recall that $0\leq \phi \leq 1$. Thus, we obtain from \eqref{G-f condn}, \eqref{G-L2-2}, \eqref{G-T2-3} and the H\"older inequality that
      \begin{align}\label{G-T2-14}
          \Bigg|\int_{K_1}&f^\epsilon(x,u_\epsilon(x), \nabla u_\epsilon(x), D_s^p u_\epsilon(x))(u-u_\epsilon)\phi dx\Bigg|\nonumber\\
          &\leq \int_{K_1}\sup\limits_{y\in B_{r(\epsilon)}(x)}|f(y,u_\epsilon(x), \nabla u_\epsilon(x), D_s^p u_\epsilon(x))|(u-u_\epsilon)\phi dx\nonumber\\
          &\leq (M_u(g_1)+M_u(g_2)+1)\Bigg(\int_{K_1} \Big(|\nabla u_\epsilon|^{\frac{p-1}{p}}+|D_s^p u_\epsilon|^{p-1}+\|g_3\|_{L^\infty(K_1)}\Big)(u-u_\epsilon)\phi dx \Bigg)\nonumber\\
          &\leq C\Bigg(\int_{K_1}|\nabla u_\epsilon|^pdx\Bigg)^\frac{p-1}{p}\Bigg(\int_{K_1}|u-u_\epsilon|^pdx\Bigg)^\frac{1}{p}\nonumber\\
          &\ \ \ +C\Bigg(\int_{K_1}\int_{\mathbb{R}^N}{|u_\epsilon(x)-u_\epsilon(y)|^p}{K_{s,p}(x,y)} dydx\Bigg)^\frac{p-1}{p}\Bigg(\int_{K_1}|u-u_\epsilon|^pdx\Bigg)^\frac{1}{p}\nonumber\\
          & \ \ \ +C\Bigg(\int_{K_1}|\phi|^\frac{p}{p-1}dx\Bigg)^\frac{p-1}{p}\Bigg(\int_{K_1}|u-u_\epsilon|^pdx\Bigg)^\frac{1}{p}\nonumber\\
          &\leq C\|u-u_\epsilon\|_{L^p(K_1)}\nonumber\\
          &\rightarrow 0 \text{ as } \epsilon\rightarrow 0.
      \end{align}
      Substituting \eqref{G-T2-14} in \eqref{G-T2-13}, we deduce
      \begin{equation}\label{G-T2-15}
          \lim\limits_{ \epsilon\rightarrow 0}\int_K|\nabla u_\epsilon-\nabla u|^pdx+\int_{K}\int_{\mathbb{R}^N}\Big|u(x)-u(y)-(u_\epsilon(x)-u_\epsilon(y))\Big|^pK_{s,p}(x,y)dydx=0. 
      \end{equation}
      Using \eqref{G-T2-15} and Lemma \ref{G-L4}, for all $v\in C_c^\infty(\Omega)$ with $\operatorname{supp}v=K$, we get that
      \begin{equation}\label{G-T2-16}
          \lim\limits_{\epsilon\rightarrow 0}\int_K f^\epsilon(x,u_\epsilon,\nabla u_\epsilon, D_s^p u_\epsilon)vdx=\int_K f(x,u,\nabla u, D_s^p u)vdx.
      \end{equation}
      Combining \eqref{G-T2-1}, \eqref{G-T2-6} and \eqref{G-T2-16}, we finally arrive at
      \begin{equation*}
          H_{s,p,\Omega}(u,v)\geq \int_{\Omega}f(x,u, \nabla u, D_s^p u)vdx,
      \end{equation*}
      for all $v\in C_c^\infty(\Omega)$ with $\operatorname{supp}v=K$. Hence the proof is complete.
  \hfill\qedsymbol{} \medskip

\noindent \textbf{Conflict of interest statement:} On behalf of the authors, the corresponding author states that there is no conflict of interest.
\newline
\textbf{Data availability statement:} Data sharing does not apply to this article as no datasets were generated or analyzed during the current study.

\section*{Acknowledgement}
The author R. Lakshmi thanks the financial support provided by the Ministry of Education (formerly known as MHRD), Government of India. SG gratefully acknowledges the financial support for this research work under ARG-MATRICS, grant No: ANRF/ARGM/2025/001570/MTR, Anusandhan National Research Foundation (ANRF), Government of India.


\begin{thebibliography}{10}

\bibitem{BM21}
B. Barrios, M. Medina,
\newblock Equivalence of weak and viscosity solutions in fractional non-homogeneous problems,
\newblock {\em Math. Ann.},
381(3-4):1979--2012, 2021.



\bibitem{BMV24}
S. Biagi, D. Mugnai, E. Vecchi,
\newblock A Brezis-Oswald approach for mixed local and nonlocal operators,
\newblock {\em Commun. Contemp. Math.},
26(2), Paper No. 2250057, 28 pp., 2024.



\bibitem{BLO20}
J. E. M. Braga, R. A. Leit\~ao, J. E. L. Oliveira,
\newblock Free boundary theory for singular/degenerate nonlinear equations with right hand side: a non-variational approach,
\newblock {\em Calc. Var. Partial Differential Equations},
59(2), Paper No. 86, 29 pp., 2020.


\bibitem{B11}
H.~Brezis.
\newblock {\em Functional analysis, {S}obolev spaces and partial differential equations}.
\newblock Springer, New York, 2011.

\bibitem{CIL92}
M. Crandall, H. Ishii, P.-L. Lions, 
\newblock User's guide to viscosity solutions of second order partial differential equations,
\newblock {\em Bull. Am. Math. Soc.}, 27:1--67, 1992.

\bibitem{DFR19}
L. M. Del Pezzo, R. Ferreira, J. D. Rossi,  
\newblock Eigenvalues for a combination between local and nonlocal {$p$}-{L}aplacians,
\newblock {\em Fract. Calc. Appl. Anal.},
22(5):1414--1436, 2019.

\bibitem{DD12}
F. Demengel, G. Demengel,
\newblock {\em Functional Spaces For The Theory Of Elliptic Partial Differential Equations},
\newblock Springer, London; EDP Sciences, Les Ulis, XVIII, 465, 
2012.


\bibitem{DKP16}
A. Di Castro, T. Kuusi, G. Palatucci,
\newblock Local behavior of fractional {$p$}-minimizers,
\newblock {\em Ann. Inst. H. Poincar\'e{} C Anal. Non Lin\'eaire},
33(5):1279--129, 2016.

\bibitem{DPV12}
E. Di Nezza, G. Palatucci, E. Valdinoci,
\newblock Hitchhiker's guide to the fractional {S}obolev spaces,
\newblock {\em Bull. Sci. Math.},
136(5):521--573, 2012.


\bibitem{FRZ24}
Y. Fang, V. D. R\u adulescu, C. Zhang,
\newblock Equivalence of weak and viscosity solutions for the nonhomogeneous double phase equation,
\newblock{\em Math. Ann.},
388(3):2519--2559, 2024.

\bibitem{FZ23}
Y. Fang, C. Zhang,
\newblock On weak and viscosity solutions of nonlocal double phase equations,
\newblock{\em Int. Math. Res. Not. IMRN},
2023(5):3746--3789, 2023.


\bibitem{GarainKinnunen}
P. Garain, J. Kinnunen,
\newblock On the regularity theory for mixed local and nonlocal quasilinear elliptic equations,
\newblock {\em Trans. Amer. Math. Soc.},
375(8):5393--5423, 2022.



\bibitem{GLZ26}
S. Ghosh, and R. Lakshmi, C. Zhang, 
\newblock {\em On equivalence of weak and viscosity solutions to nonlocal double phase problems with nonhomogeneous data}, 
Preprint arXiv:2505.16461, 2025. 


\bibitem{I95}
H. Ishii,
\newblock On the equivalence of two notions of weak solutions, viscosity solutions and distribution solutions,
\newblock {\em Funkcial. Ekvac.},
38(1):101--120, 1995.

\bibitem{IN10}
H. Ishii, G. Nakamura, 
\newblock A class of integral equations and approximation of {$p$}-{L}aplace equations,
\newblock {\em Calc. Var. Partial Differential Equations},
37(3-4):485--522, 2010.
 

\bibitem{JJ12}
V. Julin, P. Juutinen,
\newblock A new proof for the equivalence of weak and viscosity solutions for the {$p$}-{L}aplace equation,
\newblock {\em Comm. Partial Differential Equations},
37(5):934--946, 2012.

\bibitem{JLM01}
P. Juutinen, P. Lindqvist, J. J. Manfredi,
\newblock On the equivalence of viscosity solutions and weak solutions for a quasi-linear equation,
\newblock {\em SIAM J. Math. Anal.},
33(3):699--717, 2001.

\bibitem{JLP10}
P. Juutinen, T. Lukkari, M. Parviainen
\newblock Equivalence of viscosity and weak solutions for the {$p(x)$}-{L}aplacian,
\newblock {\em Ann. Inst. H. Poincar\'{e} Anal. Non Lin\'{e}aire},
27(6):1471--1487, 2010.

\bibitem{K04}
 S. Koike,
 \newblock A Beginner's Guide to the Theory of Viscosity Solutions,
 \newblock {\em MSJ Memoirs 13, Mathematical Society of Japan}, Tokyo, 2004.


\bibitem{KKL19} 
J. Korvenp\"a\"a, T. Kuusi, E. Lindgren,
\newblock Equivalence of solutions to fractional {$p$}-{L}aplace type equations,
\newblock {\em J. Math. Pures Appl. (9)},
132:1--26, 2019.


\bibitem{LGG24}
R. Lakshmi, R. Kr. Giri, S. Ghosh,
\newblock A weighted eigenvalue problem for mixed local and nonlocal operators with potential, 
\newblock {\em Math. Nachr.}, 
299(2):367--396, 2026.

\bibitem{LG2026}
R. Lakshmi, S. Ghosh,
\newblock Equivalence of weak and viscosity solutions to mixed local and nonlocal $p$-Laplace equation, 
\newblock {\em Z. Anal. Anwend.}, 2026. \url{https://doi.org/10.4171/ZAA/1815}


\bibitem{L23}
G. Leoni,
\newblock {\em A First Course In Fractional {S}obolev Spaces},
\newblock American Mathematical Society, Providence, RI, XV, 586, 2023.

\bibitem{Lind16}
P. Lindqvist,
\newblock Notes on the infinity Laplace equation, Springer, 2016.


\bibitem{Lind19}
P. Lindqvist,
\newblock Notes on the stationary $p$-Laplace equation, Springer, Berlin, 2019.

\bibitem{LM07}
P. Lindqvist, J. Manfredi, 
\newblock Viscosity solutions of the evolutionary $p$-Laplace equation,
\newblock {\em Differ. Integr. Equ.}, 20:1303--1319, 2007.

\bibitem{L83}
P.-L. Lions,
\newblock Optimal control of diffusion processes and {H}amilton-{J}acobi-{B}ellman equations. {II}. {V}iscosity solutions and uniqueness,
\newblock {\em Comm. Partial Differential Equations},
8(11):1229--1276, 1983.

\bibitem{MO19}
M. Medina, P. Ochoa,
\newblock On viscosity and weak solutions for non-homogeneous {$p$}-{L}aplace equations,
\newblock {\em Adv. Nonlinear Anal.},
8(1):468--481, 2019.

\bibitem{MO23}
M. Medina, P. Ochoa,
\newblock Equivalence of solutions for non-homogeneous {$p(x)$}-{L}aplace equations,
\newblock {\em Math. Eng.},
5(2), Paper No. 044, 19 pp., 2023.

\bibitem{SZ23}
B. Shang, C. Zhang,
\newblock A strong maximum principle for mixed local and nonlocal {$p$}-{L}aplace equations,
\newblock {\em Asymptot. Anal.},
133(1-2):1--12, 2023.

\bibitem{S18}
J. Siltakoski,
\newblock Equivalence of viscosity and weak solutions for the normalized {$p(x)$}-{L}aplacian,
\newblock {\em Calc. Var. Partial Differential Equations},
57(4), Paper No. 95, 20 pp., 2018.



\end{thebibliography}
\end{document}